\documentclass[10pt]{amsart}

\usepackage{enumerate}
\usepackage{amsmath,amssymb}

\newtheorem{theorem}{Theorem}
\newtheorem{lemma}[theorem]{Lemma}

\theoremstyle{definition}
\newtheorem{remark}[theorem]{Remark}

\begin{document}

\title[Scottish Book Problem 155]{On the Scottish Book Problem 155\\
by Mazur and Sternbach}

\author[M. Mori]{Michiya Mori}

\dedicatory{Dedicated to Professor Yasuyuki Kawahigashi on the occasion of his $60+1$st birthday.}

\address{Graduate School of Mathematical Sciences, The University of Tokyo, 3-8-1 Komaba, Meguro-ku, Tokyo, 153-8914, Japan; Interdisciplinary Theoretical and Mathematical Sciences Program (iTHEMS), RIKEN, 2-1 Hirosawa, Wako, Saitama, 351-0198, Japan.}
\email{mmori@ms.u-tokyo.ac.jp}

\thanks{The author is supported by JSPS KAKENHI Grant Number 22K13934.}
\subjclass[2020]{Primary 46B04.} 

\keywords{Scottish Book; Banach space; isometry}

\date{}

\begin{abstract}
Problem 155 of the Scottish Book asks whether every bijection $U\colon X\to Y$ between two Banach spaces $X, Y$ with the property that, each point of $X$ has a neighborhood on which $U$ is isometric, is globally isometric on $X$. 
We prove that this is true under the additional assumption that $X$ is separable and the weaker assumption of surjectivity instead of bijectivity.
\end{abstract}

\maketitle
\thispagestyle{empty}

According to \cite{S}, Problem 155 of the famous Scottish Book posed by Mazur and Sternbach on November 18 in 1936 asks the following.
\begin{quote}
Given are two spaces $X, Y$ of type (B), $y = U(x)$ is a one-to-one mapping of the
space $X$ onto the whole space $Y$ with the following property: For every $x_0 \in X$ there
exists an $\varepsilon>0$ such that the mapping $y = U(x)$, considered for $x$ belonging to the
sphere with the center $x_0$ and radius $\varepsilon$, is an isometric mapping. Is the mapping
$y = U(x)$ an isometric transformation?
\end{quote}

Let us rephrase this problem in a modern manner. 
For a point $x_0$ in a Banach space $X$ and a positive real number $r$, we use the symbol $B(x_0,r):=\{x\in X\mid \lVert x-x_0\rVert \leq r\}$ for the ball with center $x_0$ and radius $r$.  
Note that a \emph{sphere} in the Scottish Book usually means a ball in this sense. (A set of the form $S(x_0, r):=\{x\in X\mid \lVert x-x_0\rVert =r\}$ is called the \emph{surface of a sphere} in the Scottish Book. If we interpret that the ``sphere'' in the above problem means $S(x_0, r)$, then one may encounter simple counterexamples, see Remark \ref{r}.)

In what follows, we assume that $X, Y$ are two real Banach spaces ($=$ spaces of type (B)). 
A mapping $f\colon D\to E$  between subsets $D\subset X$, $E\subset Y$ is called an \emph{isometry} if it satisfies $\lVert f(x_1)-f(x_2)\rVert =\lVert x_1-x_2\rVert$ for any $x_1, x_2\in D$. 
The above problem translates into the following.
\begin{quote}
Let $U\colon X\to Y$ be a bijection with the following property: 
For every $x_0 \in X$ there exists an $\varepsilon>0$ such that the mapping $U$ restricted to $B(x_0, \varepsilon)$ is an isometry. 
Is $U\colon X\to Y$ an isometry on $X$?
\end{quote}

One may find no commentaries on Problem 155 in the second printed edition of the Scottish Book \cite{S} published in 2015. 
It seems that the solution to this problem is still unknown. 
The purpose of this note is to give a positive solution to this problem under the additional assumption that $X$ is separable and the weaker assumption of surjectivity instead of bijectivity.
The nonseparable case remains open.

The celebrated Mazur--Ulam theorem \cite{MU} asserts that every surjective isometry between two real Banach spaces is affine.
A version of the Mazur--Ulam theorem due to Mankiewicz \cite{M} asserts that any surjective isometry between balls of Banach spaces extends uniquely to an affine surjective isometry of the whole Banach spaces.
We give lemmas that are relevant to the Mazur--Ulam and Mankiewicz's theorems.

\begin{lemma}\label{l1}
Let $x_0\in X$ and $r>0$.
Let a subset $D\subset X$ satisfy $B(x_0, r)\subset D$. 
If $g\colon D\to Y$ is an isometry satisfying $g(x)=x$ for every $x\in B(x_0, r)$, then $g(x)=x$ holds for every $x\in D\cap B(x_0, 2r)$.
\end{lemma}
\begin{proof}
We imitate the argument in the original proof of the Mazur--Ulam theorem \cite{MU}. 
For a bounded subset $C\subset X$, we define $\delta(C):=\sup \{\lVert x-x'\rVert\mid x, x'\in C\}$. 
Let $x_1\in D\cap B(x_0, 2r)$. 
Then $C_1:=B(x_0, \lVert x_1-x_0\rVert/2) \cap B(x_1, \lVert x_1-x_0\rVert/2)$ is a bounded subset with $C_1\subset B(x_0, r)$. 
For $n\geq 1$, we inductively define $C_{n+1}:= \{x\in C_n\mid C_n\subset B(x, \delta(C_n)/2)\}$. 
Then it clearly follows that $C_1\supset C_2\supset \cdots$ and $\delta(C_n)\leq 2^{-1}\delta(C_{n-1})\leq \cdots\leq 2^{-n+1} \delta(C_1)$. 
Therefore, we see that $\bigcap_{n\geq 1}C_n$ consists of at most one point. 
We show that $\{x_2\}= \bigcap_{n\geq 1}C_n$, where $x_2=(x_0+x_1)/2$. 
By symmetricity, for every $n$, we find that a point $x$ lies in $C_n$ if and only if $2x_2-x\in C_n$. 
It is clear that $x_2\in C_1$. 
Assume that $n\geq 1$ and $x_2\in C_n$. 
If $x\in C_n$, then $2x_2-x\in C_n$, thus $2\lVert x_2-x\rVert = \lVert (2x_2-x)-x\rVert\leq \delta(C_n)$. It follows that $x_2\in C_{n+1}$. 
Thus we obtain $\{x_2\}= \bigcap_{n\geq 1}C_n$. 

Since $g$ is an isometry and $g(x_0)=x_0$, we see that 
$\lVert g(x_1)-x_0\rVert=\lVert g(x_1)-g(x_0)\rVert=\lVert x_1-x_0\rVert$. 
Set $\widehat{C}_1:=B(x_0, \lVert x_1-x_0\rVert/2) \cap B(g(x_1), \lVert x_1-x_0\rVert/2)$ and $\widehat{C}_{n+1}:= \{x\in \widehat{C}_n\mid \widehat{C}_n\subset B(x, \delta(\widehat{C}_n)/2)\}$ for $n\geq 1$. 
Then exactly the same argument as in the preceding paragraph shows that $\{(x_0+g(x_1))/2\}= \bigcap_{n\geq 1}\widehat{C}_n$.
Observe that our assumption implies $\widehat{C}_n=g(C_n)$ for every $n\geq 1$.
It follows that $(x_0+g(x_1))/2= g(x_2)$. 
Since $x_2\in C_1\subset B(x_0, r)$, we obtain $g(x_2)=x_2=(x_0+x_1)/2$. 
Thus $g(x_1)=x_1$.
\end{proof}

\begin{lemma}\label{l2}
Let $x_0\in X$ and $r>0$.
If $f\colon B(x_0, r)\to Y$ is an isometry and $f(B(x_0, r))$ has nonempty interior, then $f$ extends uniquely to an affine surjective isometry from $X$ onto $Y$.
\end{lemma}
\begin{proof}
We may assume $x_0=0\in X$ and $r=1$ without loss of generality. 
Take $y_1$ in the interior of $f(B(0, 1))$. 
Set $x_1:= f^{-1}(y_1)$. 
Since $f$ is an isometry, there is $r_1>0$ such that $B(x_1, r_1)\subset B(0, 1)$ and $f(B(x_1, r_1))=B(y_1, r_1)$. 
It follows from Mankiewicz theorem that there is an affine surjective isometry $g\colon X\to Y$ such that $f=g$ on $B(x_1, r_1)$. 
Then the mapping $F:= g^{-1}\circ f\colon B(0, 1)\to X$ is an isometry satisfying $F(x)=x$ for every $x\in B(x_1, r_1)$. 

We consider the set $D$ of all $z\in B(0, 1)$ with the property that there is some $\rho>0$ such that $B(z,\rho)\subset B(0, 1)$ and $F(x)=x$ for every $x\in B(z, \rho)$. 
Observe that $x_1\in D$. 
Assume that $z\in D$. 
Let $m$ be the smallest positive integer with $B(z, 2^m\rho)\not\subset B(0, 1)$.
Applying Lemma \ref{l1} (with center $z$) repeatedly, we obtain $F(x)=x$ for every $x\in B(z, 2^m\rho)\cap B(0, 1)$.
In particular, for $z=\lambda z_0\in D$ with $\lVert z_0\rVert =1$ and $0\leq \lambda<1$, we see $\mu z_0\in D$ for every $\mu\in (2\lambda-1, 1)$. 
Beginning with $z=x_1$ and using this repeatedly, we obtain $0\in D$. 
Then another repeated application of Lemma \ref{l1} (with center $0$) shows that $F(x)=x$ for every $x\in B(0, 1)$.
\end{proof}

Now we are ready to prove the main theorem of this note.
\begin{theorem}
Assume that $X$ is separable. 
Let $U\colon X\to Y$ be a surjection with the following property: 
For every $x_0 \in X$ there exists an $\varepsilon_{x_0}>0$ such that the mapping $U$ restricted to $B(x_0, \varepsilon_{x_0})$ is an isometry. 
Then $U\colon X\to Y$ is an 
isometry on $X$.
\end{theorem}
\begin{proof}
Recall that a separable metric space is Lindel\"{o}f \cite[Theorem 16.11]{W} (every open cover has countable subcover). 
Since the family of open sets $\{\{x\in X\mid \lVert x-x_0\rVert<\varepsilon_{x_0}\}\mid x_0\in X\}$ covers the separable metric space $X$, one may find a sequence $x_1, x_2, \ldots$ in $X$ such that $X\subset \bigcup_{n\geq 1}\{x\in X\mid \lVert x-x_n\rVert<\varepsilon_{x_n}\}$. 
For each $n$, $U$ restricts to an isometry on $B(x_n, \varepsilon_{x_n})$, and thus its image $U(B(x_n, \varepsilon_{x_n}))$, being complete, is closed in $Y$. 
Since $U$ is surjective, we see that $Y= \bigcup_{n\geq 1}U(B(x_n, \varepsilon_{x_n}))$. 
Thus $Y$ is written as a countable union of closed subsets.
By the Baire Category Theorem, there is some $m\geq 1$ such that $U(B(x_m, \varepsilon_{x_m}))$ has nonempty interior. 
Thus Lemma \ref{l2} implies that the restriction of $U$ to $B(x_m, \varepsilon_{x_m})$ extends to a surjective isometry $f\colon X\to Y$. 

Let us consider the nonempty open set $D$ that consists of all $z\in X$ with the property that $U(x)=f(x)$ for every $x$ in some neighborhood of $z$.  
Assume that $D\neq X$. 
Since $X$ is connected, we may take a boundary point $x_0$ of $D$. 
Then $U$ restricted to $B(x_0, \varepsilon_{x_0})$ is an isometry, and $B(x_0, \varepsilon_{x_0})\cap D$ has nonempty interior.
Since $U=f$ on $D$ and $f$ is an affine surjective isometry, we see from Lemma \ref{l2} that $U=f$ on $B(x_0, \varepsilon_{x_0})$. 
This leads to $x_0\in D$, a contradiction.
\end{proof}
 
\begin{remark}\label{r}
Let $U\colon X\to Y$ be a bijection with the following property: 
For every $x_0 \in X$ there exists an $\varepsilon>0$ such that the mapping $U$ restricted to $S(x_0, \varepsilon)$ is an isometry. 
Is the mapping $U\colon X\to Y$ an isometry on $X$?
The answer is NO. 
To see this, consider an arbitrary $X$ and assume $X=Y$. 
Let $U\colon X\to X$ be ANY mapping with the properties that $U$ restricts to a bijection on $B(0, 1)$ and that $U(x)=x$ for every $x\in X\setminus B(0, 1)$. 
Even if $U$ behaves very wildly on $B(0, 1)$, $U$ clearly satisfies the assumption of this question.
Note that in the Scottish Book it is remarked that 
the answer to Problem 155 is affirmative if we additionally assume $U^{-1}$ is continuous, and this is automatic when $\dim Y<\infty$ or $Y$ is strictly convex. 
Therefore, the ``sphere'' in Problem 155 by no means stands for the set $S(x_0, \varepsilon)$.
\end{remark}

\medskip\medskip

\noindent 
\textbf{Acknowledgements.}\,
The author appreciates Kan Kitamura (University of Tokyo) for leading the author to the correct interpretation of the term ``sphere'' in Problem 155 at the conference \emph{Operator Algebras and Mathematical Physics} in July 2023 at the University of Tokyo, which celebrates Professor Yasuyuki Kawahigashi's 60th birthday.

\end{document}